\documentclass{amsart}

\usepackage{amssymb} \usepackage{amsbsy,
  amsthm, amsmath,amstext, amsopn} \usepackage[all]{xy}
\usepackage{amsfonts} \usepackage{amscd} \usepackage{stmaryrd}
\newtheorem{thm}{Theorem} [section]
\newtheorem{lemma}[thm]{Lemma}

\newtheorem{corollary}[thm]{Corollary}

\theoremstyle{definition}

\newtheorem{defn}[thm]{Definition}

\newtheorem{example}[thm]{Example}

\theoremstyle{remark}

\newtheorem{remark}[thm]{Remark}

\begin{document}

\numberwithin{equation}{section}

\newcommand{\hs}{\mbox{\hspace{.4em}}}
\newcommand{\ds}{\displaystyle}
\newcommand{\bd}{\begin{displaymath}}
\newcommand{\ed}{\end{displaymath}}
\newcommand{\bcd}{\begin{CD}}
\newcommand{\ecd}{\end{CD}}

\newcommand{\on}{\operatorname}
\newcommand{\proj}{\operatorname{Proj}}
\newcommand{\bproj}{\underline{\operatorname{Proj}}}

\newcommand{\spec}{\operatorname{Spec}}
\newcommand{\Spec}{\operatorname{Spec}}
\newcommand{\bspec}{\underline{\operatorname{Spec}}}
\newcommand{\pline}{{\mathbf P} ^1}
\newcommand{\aline}{{\mathbf A} ^1}
\newcommand{\pplane}{{\mathbf P}^2}
\newcommand{\aplane}{{\mathbf A}^2}
\newcommand{\coker}{{\operatorname{coker}}}
\newcommand{\ldb}{[[}
\newcommand{\rdb}{]]}

\newcommand{\Sym}{\operatorname{Sym}^{\bullet}}
\newcommand{\Symp}{\operatorname{Sym}}
\newcommand{\Pic}{\on{Pic}}
\newcommand{\Aut}{\operatorname{Aut}}
\newcommand{\PAut}{\operatorname{PAut}}

\newcommand{\too}{\twoheadrightarrow}
\newcommand{\C}{{\mathbf C}}
\newcommand{\Z}{{\mathbf Z}}
\newcommand{\Q}{{\mathbf Q}}
\newcommand{\R}{{\mathbf R}}
\newcommand{\Cx}{{\mathbf C}^{\times}}
\newcommand{\Cbar}{\overline{\C}}
\newcommand{\Cxbar}{\overline{\Cx}}
\newcommand{\cA}{{\mathcal A}}
\newcommand{\cS}{{\mathcal S}}
\newcommand{\cV}{{\mathcal V}}
\newcommand{\cM}{{\mathcal M}}
\newcommand{\bA}{{\mathbf A}}
\newcommand{\cB}{{\mathcal B}}
\newcommand{\cC}{{\mathcal C}}
\newcommand{\cD}{{\mathcal D}}
\newcommand{\D}{{\mathcal D}}
\newcommand{\cs}{{\mathbf C} ^*}
\newcommand{\boldc}{{\mathbf C}}
\newcommand{\cE}{{\mathcal E}}
\newcommand{\cF}{{\mathcal F}}
\newcommand{\bF}{{\mathbf F}}
\newcommand{\cG}{{\mathcal G}}
\newcommand{\G}{{\mathbb G}}
\newcommand{\cH}{{\mathcal H}}
\newcommand{\CI}{{\mathcal I}}
\newcommand{\cJ}{{\mathcal J}}
\newcommand{\cK}{{\mathcal K}}
\newcommand{\cL}{{\mathcal L}}
\newcommand{\baL}{{\overline{\mathcal L}}}
\newcommand{\M}{{\mathcal M}}
\newcommand{\Mf}{{\mathfrak M}}
\newcommand{\bM}{{\mathbf M}}
\newcommand{\bm}{{\mathbf m}}
\newcommand{\cN}{{\mathcal N}}
\newcommand{\theo}{\mathcal{O}}
\newcommand{\cP}{{\mathcal P}}
\newcommand{\cR}{{\mathcal R}}
\newcommand{\Pp}{{\mathbb P}}
\newcommand{\boldp}{{\mathbf P}}
\newcommand{\boldq}{{\mathbf Q}}
\newcommand{\bbL}{{\mathbf L}}
\newcommand{\cQ}{{\mathcal Q}}
\newcommand{\cO}{{\mathcal O}}
\newcommand{\cT}{{\mathcal T}}
\newcommand{\Oo}{{\mathcal O}}
\newcommand{\cY}{{\mathcal Y}}
\newcommand{\OX}{{\Oo_X}}
\newcommand{\OY}{{\Oo_Y}}
\newcommand{\cZ}{{\mathcal Z}}
\newcommand{\otY}{{\underset{\OY}{\ot}}}
\newcommand{\otX}{{\underset{\OX}{\ot}}}
\newcommand{\cU}{{\mathcal U}}\newcommand{\cX}{{\mathcal X}}
\newcommand{\cW}{{\mathcal W}}
\newcommand{\boldz}{{\mathbf Z}}
\newcommand{\qgr}{\operatorname{q-gr}}
\newcommand{\gr}{\operatorname{gr}}
\newcommand{\rk}{\operatorname{rk}}
\newcommand{\Sh}{\operatorname{Sh}}
\newcommand{\SH}{{\underline{\operatorname{Sh}}}}
\newcommand{\End}{\operatorname{End}}
\newcommand{\uEnd}{\underline{\operatorname{End}}}
\newcommand{\Hom}{\operatorname{Hom}}
\newcommand{\uHom}{\underline{\operatorname{Hom}}}
\newcommand{\uHomY}{\uHom_{\OY}}
\newcommand{\uHomX}{\uHom_{\OX}}
\newcommand{\Ext}{\operatorname{Ext}}
\newcommand{\bExt}{\operatorname{\bf{Ext}}}
\newcommand{\Tor}{\operatorname{Tor}}

\newcommand{\inv}{^{-1}}
\newcommand{\airtilde}{\widetilde{\hspace{.5em}}}
\newcommand{\airhat}{\widehat{\hspace{.5em}}}
\newcommand{\nt}{^{\circ}}
\newcommand{\del}{\partial}

\newcommand{\supp}{\operatorname{supp}}
\newcommand{\GK}{\operatorname{GK-dim}}
\newcommand{\hd}{\operatorname{hd}}
\newcommand{\id}{\operatorname{id}}
\newcommand{\res}{\operatorname{res}}
\newcommand{\lrar}{\leadsto}
\newcommand{\im}{\operatorname{Im}}
\newcommand{\HH}{\operatorname{H}}
\newcommand{\TF}{\operatorname{TF}}
\newcommand{\Bun}{\operatorname{Bun}}

\newcommand{\F}{\mathcal{F}}
\newcommand{\Ff}{\mathbb{F}}
\newcommand{\nthord}{^{(n)}}
\newcommand{\Gr}{{\mathfrak{Gr}}}

\newcommand{\Fr}{\operatorname{Fr}}
\newcommand{\GL}{\operatorname{GL}}
\newcommand{\gl}{\mathfrak{gl}}
\newcommand{\SL}{\operatorname{SL}}
\newcommand{\ff}{\footnote}
\newcommand{\ot}{\otimes}
\def\Ext{\operatorname {Ext}}
\def\Hom{\operatorname {Hom}}
\def\Ind{\operatorname {Ind}}
\def\bbZ{{\mathbb Z}}

\newcommand{\nc}{\newcommand}
\nc{\ol}{\overline} \nc{\cont}{\on{cont}} \nc{\rmod}{\on{mod}}
\nc{\Mtil}{\widetilde{M}} \nc{\wb}{\overline} \nc{\wt}{\widetilde}
\nc{\wh}{\widehat} \nc{\sm}{\setminus} \nc{\mc}{\mathcal}
\nc{\mbb}{\mathbb}  \nc{\K}{{\mc K}} \nc{\Kx}{{\mc K}^{\times}}
\nc{\Ox}{{\mc O}^{\times}} \nc{\unit}{{\bf \on{unit}}}
\nc{\boxt}{\boxtimes} \nc{\xarr}{\stackrel{\rightarrow}{x}}

\nc{\Ga}{\G_a}
 \nc{\PGL}{{\on{PGL}}}
 \nc{\PU}{{\on{PU}}}

\nc{\h}{{\mathfrak h}} \nc{\kk}{{\mathfrak k}}
 \nc{\Gm}{\G_m}
\nc{\Gabar}{\wb{\G}_a} \nc{\Gmbar}{\wb{\G}_m} \nc{\Gv}{G^\vee}
\nc{\Tv}{T^\vee} \nc{\Bv}{B^\vee} \nc{\g}{{\mathfrak g}}
\nc{\gv}{{\mathfrak g}^\vee} \nc{\RGv}{\on{Rep}\Gv}
\nc{\RTv}{\on{Rep}T^\vee}
 \nc{\Flv}{{\mathcal B}^\vee}
 \nc{\TFlv}{T^*\Flv}
 \nc{\Fl}{{\mathfrak Fl}}
\nc{\RR}{{\mathcal R}} \nc{\Nv}{{\mathcal{N}}^\vee}
\nc{\St}{{\mathcal St}} \nc{\ST}{{\underline{\mathcal St}}}
\nc{\Hec}{{\bf{\mathcal H}}} \nc{\Hecblock}{{\bf{\mathcal
H_{\alpha,\beta}}}} \nc{\dualHec}{{\bf{\mathcal H^\vee}}}
\nc{\dualHecblock}{{\bf{\mathcal H^\vee_{\alpha,\beta}}}}
\newcommand{\ramBun}{{\bf{Bun}}}
\newcommand{\ramBuno}{\ramBun^{\circ}}

\nc{\Buntheta}{{\bf Bun}_{\theta}} \nc{\Bunthetao}{{\bf
Bun}_{\theta}^{\circ}} \nc{\BunGR}{{\bf Bun}_{G_\R}}
\nc{\BunGRo}{{\bf Bun}_{G_\R}^{\circ}}
\nc{\HC}{{\mathcal{HC}}}
\nc{\risom}{\stackrel{\sim}{\to}} \nc{\Hv}{{H^\vee}}
\nc{\bS}{{\mathbf S}}
\def\Rep{\operatorname {Rep}}
\def\Conn{\operatorname {Conn}}

\nc{\Vect}{{\operatorname{Vect}}}
\nc{\Hecke}{{\operatorname{Hecke}}}

\newcommand{\ZZ}{{Z_{\bullet}}}
\nc{\HZ}{{\mc H}\ZZ} \nc{\eps}{\epsilon}

\nc{\CN}{\mathcal N} \nc{\BA}{\mathbb A}
\nc{\XYX}{X\times_Y X}

\nc{\ul}{\underline}

\nc{\bn}{\mathbf n} \nc{\Sets}{{\on{Sets}}} \nc{\Top}{{\on{Top}}}

\nc{\Simp}{{\mathbf \Delta}} \nc{\Simpop}{{\mathbf\Delta^\circ}}

\nc{\Cyc}{{\mathbf \Lambda}} \nc{\Cycop}{{\mathbf\Lambda^\circ}}

\nc{\Mon}{{\mathbf \Lambda^{mon}}}
\nc{\Monop}{{(\mathbf\Lambda^{mon})\circ}}

\nc{\Aff}{{\on{Aff}}} \nc{\Sch}{{\on{Sch}}}

\nc{\bul}{\bullet}
\nc{\module}{{\operatorname{-mod}}}

\nc{\dstack}{{\mathcal D}}

\nc{\BL}{{\mathbb L}}

\nc{\BD}{{\mathbb D}}

\nc{\BR}{{\mathbb R}}

\nc{\BT}{{\mathbb T}}

\nc{\SCA}{{\mc{SCA}}}
\nc{\DGA}{{\mc DGA}}

\nc{\DSt}{{DSt}}

\nc{\lotimes}{{\otimes}^{\mathbf L}}

\nc{\bs}{\backslash}

\nc{\Lhat}{\widehat{\mc L}}

\newcommand{\Coh}{{\on{Coh}}}

\nc{\QCoh}{QC}
\nc{\QC}{QC}
\nc\Perf{\on{Perf}}
\nc{\Cat}{{\on{Cat}}}
\nc{\dgCat}{{\on{dgCat}}}
\nc{\bLa}{{\mathbf \Lambda}}

\nc{\RHom}{\mathbf{R}\hspace{-0.15em}\on{Hom}}
\nc{\REnd}{\mathbf{R}\hspace{-0.15em}\on{End}}
\nc{\colim}{\on{colim}}
\nc{\oo}{\infty}
\nc\Mod{\on{Mod}}

\nc\fh{\mathfrak h}
\nc\al{\alpha}
\nc\la{\alpha}
\nc\BGB{B\bs G/B}
\nc\QCb{QC^\flat}
\nc\qc{\cQ}

\nc{\fg}{\mathfrak g}

\nc{\fn}{\mathfrak n}
\nc{\Map}{\on{Map}} \nc{\fX}{\mathfrak X}

\nc{\ch}{\check}
\nc{\fb}{\mathfrak b} \nc{\fu}{\mathfrak u} \nc{\st}{{st}}
\nc{\fU}{\mathfrak U}
\nc{\fZ}{\mathfrak Z}

\nc\fk{\mathfrak k} \nc\fp{\mathfrak p}

\nc{\RP}{\mathbf{RP}} \nc{\rigid}{\text{rigid}}
\nc{\glob}{\text{glob}}

\nc{\cI}{\mathcal I}

\nc{\La}{\mathcal L}

\nc{\quot}{/\hspace{-.25em}/}

\nc\aff{\it{aff}}
\nc\BS{\mathbb S}

\nc\Loc{\on{Loc}}
\nc\Tr{{\mc Tr}}
\nc\Ch{{\mc Ch}}

\nc\git{/\hspace{-0.2em}/}
\nc{\fc}{\mathfrak c}
\nc\BC{\mathbb C}
\nc\BZ{\mathbb Z}

\nc\stab{\text{\it st}}
\nc\Stab{\text{\it St}}

\nc\perf{\on{-perf}}

\nc\intHom{\mathcal{H}om}

\nc\gtil{\widetilde\fg}

\nc\mon{\text{\it mon}}
\nc\bimon{\text{\it bimon}}

\def\w{\wedge}

\def\inv{/ \hspace{-0.25em} /}
\def\univ{\mathit{univ}}

\def\aff{\mathit{aff}}
\def\cusp{\mathit{cusp}}
\def\node{\mathit{node}}
\def\fin{\mathit{fin}}


\title[Elliptic Springer theory]{Elliptic Springer theory}

\author{David Ben-Zvi} \address{Department of Mathematics\\University
  of Texas\\Austin, TX 78712-0257} \email{benzvi@math.utexas.edu}
\author{David Nadler} \address{Department of Mathematics\\University
  of California\\Berkeley, CA 94720-3840}
\email{nadler@math.berkeley.edu}

\begin{abstract}
We introduce an elliptic version of the Grothendieck-Springer sheaf and establish elliptic analogues of the basic results of Springer theory.
From a geometric perspective, our constructions specialize geometric Eisenstein series to the resolution of degree zero, semistable $G$-bundles by degree zero $B$-bundles over an elliptic curve $E$. 
From a representation theory perspective, they produce a full embedding of representations of the elliptic or double affine Weyl group into perverse sheaves with nilpotent characteristic variety on the moduli of $G$-bundles over $E$.
The resulting objects are principal series examples of elliptic character sheaves, 
objects expected to play the role of character sheaves for loop groups.
\end{abstract}

\maketitle


\tableofcontents


\section{Introduction}

Let $G$ be a  connected reductive complex algebraic group, $B\subset G$ a Borel subgroup, $N\subset B$ its unipotent radical, and $H=B/N$
the universal Cartan. Let $\fg$, $\fb$, $\fn$, and $\fh$ denote the corresponding Lie algebras. Let $W$ denote the Weyl group.

\subsection{Springer theory}
Let us begin by recalling the standard sheaf-theoretic formulations \cite{L, BoMac, KL}
 of Springer theory \cite{Spr76, Spr78, Spr82}. We will emphasize an interpretation in terms of geometric Eisenstein series~\cite{Laumon, BG} but in the  non-standard setting of singular curves of arithmetic genus one. (See \cite{Fukaya} for a parallel symplectic treatment of
 Springer theory, and \cite{solitons} for a survey of moduli of bundles on cubic curves).
 
\subsubsection{Rational setting} Let $X_\fg$ be the flag variety of Borel subalgebras (or equivalently, subgroups), and $\cN \subset \fg$ the nilpotent cone.
 The Springer resolution 
 $$
 \xymatrix{
 \mu_\cN:\wt \cN \simeq T^* X_\fg  \simeq \{ v\in \fb\} \subset \cN \times X_{\fg}\ar[r] & \cN
 }
 $$ is a semi-small map and so the Springer sheaf $S_\cN = \mu_{\cN!}\BC_{\wt \cN}[\dim\cN]$ is perverse. The endomorphisms of $S_\cN$ as a perverse sheaf are  equivalent to the group algebra $\BC[W]$ of the Weyl group.
 The Grothendieck-Springer resolution 
 $$
 \xymatrix{
 \mu_\fg:\wt \fg \simeq \{ v\in \fb \}  \subset \fg \times X_\fg \ar[r] & \fg
 }
 $$
 is a small map and its restriction to the regular semisimple locus $\fg^{rs} \subset \fg$ is the natural $W$-cover. Hence the Grothendieck-Springer sheaf $S_\fg = \mu_{\fg!}\BC_{\wt \fg}[\dim\fg]$ is the middle-extension of 
 the natural $\BC[W]$-regular local system over $\fg^{rs} \subset \fg$. Thus the endomorphisms of $S_\fg$ as a perverse sheaf are  equivalent to the group algebra $\BC[W]$. Furthermore, under any invariant isomorphism $\fg\simeq \fg^*$,  
one can identify $S_\cN$ and $S_\fg$ as  Fourier transforms of each other \cite{Ginz, HK}.
Thus 
Springer theory compatibly realizes representations of $W$
 as perverse sheaves on $\cN$ and $\fg$.
 
 Taking quotients by the adjoint action of $G$, one can identify the equivariant Grothendieck-Springer resolution with the natural induction map of adjoint quotients
 $$
 \xymatrix{
\mu_\fg:\wt \cN/G \simeq \fb/B  \ar[r] & \fg/G
 }
 $$
For a cuspidal cubic curve $E_\cusp \simeq {\mathbb G}_a \cup \{\infty\}$ (or any simply-connected projective curve of arithmetic genus one),  this admits the interpretation as the induction map 
  $$
 \xymatrix{
\mu_\fg: \Bun^{0}_B(E_\cusp) \ar[r] & \Bun_G^{ss, 0}(E_\cusp)
}
$$
from
  degree zero $B$-bundles to degree zero, semistable $G$ bundles.  Any degree zero, semistable bundle pulls back to the trivial bundle along the normalization map $\mathbb P^1\to E_\cusp$, and the Lie algebra appears as  descent data at the cusp $\{\oo\}$. 
The correspondence of adjoint quotients
   $$
 \xymatrix{
\fh/H & \ar[l]_-{\nu_\fg}  \fb/B  \ar[r]^-{\mu_\fg} & \fg/G
 }
 $$
admits the interpretation as the correspondence of moduli of bundles
  $$
 \xymatrix{
 \Bun^0_H(E_\cusp) & \ar[l]_-{\nu_\fg} \Bun^{0}_B(E_\cusp) \ar[r]^-{\mu_\fg} & \Bun_G^{ss, 0}(E_\cusp)
}
$$ 
The $W$-action on the Grothendieck-Springer sheaf $S_\fg = \mu_{\fg!} \nu_\fg^*\BC_{\fh/H}[\dim \fg]$ reflects the functional equation of the geometric Eisenstein series construction applied to the constant sheaf.

\begin{remark}
In addition to the endomorphisms of $S_\fg$ as a perverse sheaf, the differential graded algebra of  endomorphisms of $S_\fg$ as an equivariant  complex can be calculated
$$
\End_{D^\flat_c(\fg/G)}(S_\fg) \simeq H^*(BH) \rtimes \BC[W] \simeq \on{Sym}^*(\fh^\vee[-2]) \rtimes \BC[W]
$$ 
Thus the full subcategory of $D^\flat_c(\fg/G)$ generated by $S_\fg$ is equivalent to 
 the category of finitely-generated $H^*(BH) \rtimes \BC[W]$-modules.
 (See~\cite{Rider} for a definitive account and mixed version.)

Going further,  the category of finitely-generated $H^*(BH) \rtimes \BC[W]$-modules is equivalent to
 the full subcategory of $D^\flat_c((\fh/H)/W)$ generated by the pushforward $q_!\BC_{\fh/H}$  of the constant sheaf along 
the natural quotient map 
$$
\xymatrix{
q:\fh/H\ar[r] & (\fh/H)/W
}
$$
The geometric Eisenstein series construction $\mu_{\fg!}\nu_{\fg}^*[\dim\fg]$ descends to a fully faithful embedding
$$
\xymatrix{
D^\flat_c((\fh/H)/W) \ar@{^(->}[r] & D^\flat_c(\fg/G)
}
$$ 
which recovers  Springer theory on the full subcategory
generated by $q_!\BC_{\fh/H}$.
(See~\cite{sam} where the theory is in fact established in the group-theoretic setting and for all $\D$-modules.) 
\end{remark}

 
\subsection{Trigonometric setting} Now let us expand our scope to the  group-theoretic Grothendieck-Springer resolution 
 $$
 \xymatrix{
 \mu_G:\wt G   \simeq \{ g\in B \}  \subset G \times X_\fg \ar[r] & G
 }
 $$
 One recovers the  linear Grothendieck-Springer resolution $\mu_{\fg}:\wt\fg\to\fg$  by deforming to the normal cone of the identity of $G$.
Taking quotients by the adjoint action of $G$,
we obtain the  natural induction map of adjoint quotients
 $$
 \xymatrix{
\mu_G:\wt G/G \simeq B/B \ar[r] & G/G
 }
 $$
 For a nodal cubic curve $E_\node\simeq {\mathbb G}_m \cup \{\infty\}$ (or any projective curve of arithmetic genus one with fundamental group  $\BZ$), 
 this admits the interpretation as the induction map 
 $$
 \xymatrix{
\mu_G: \Bun^{0}_B(E_\node) \ar[r] & \Bun_G^{ss, 0}(E_\node)
}
$$
from  degree zero $B$-bundles to degree zero, semistable $G$-bundles. Any degree zero, semistable bundle pulls back to the trivial bundle along the normalization map $\mathbb P^1\to E_\node$, and the group appears as descent data identifying the two preimages of the node $\{\oo\}$.
The correspondence of adjoint quotients
   $$
 \xymatrix{
H/H & \ar[l]_-{\nu_G}  B/B  \ar[r]^-{\mu_G} & G/G
 }
 $$
admits the interpretation as  the correspondence of moduli of bundles
  $$
 \xymatrix{
 \Bun^0_H(E_\node) & \ar[l]_-{\nu_G} \Bun^{0}_B(E_\node) \ar[r]^-{\mu_G} & \Bun_G^{ss, 0}(E_\node)
}
$$ 

Let us focus on the geometric Eisenstein series construction 
$$
\xymatrix{
\mu_{G!}\nu_{G}^*[\dim G]:D_c^\flat(H/H) \ar[r] & D_c^\flat(G/G)
}
$$
applied to $W$-equivariant local systems on $H$.

The fundamental group $\pi_1(H)$ is the coweight lattice $\Lambda_H = \Hom(\Gm, H)$ with spectrum the dual torus $H^\vee = \Spec \BC[\Lambda_H]$.  Thus finite-rank $W$-equivariant local systems on $H$ correspond to 
finite-dimensional representations of the affine Weyl group $W_{\aff} = \Lambda_H\rtimes W$ which in turn correspond to 
$W$-equivariant coherent sheaves on $H^\vee$ with finite support.

Starting from a finite-rank $W$-equivariant local system $\cL$,  the corresponding Grothendieck-Springer sheaf 
$$
S_{G, \cL} = \mu_{G!} \nu_G^*  \cL[\dim G]
$$
is a perverse sheaf 
with endomorphisms  
$$
\End_{\on{Perv}(G/G)} (S_{G, \cL}) \simeq \End_{\Loc(H)}(\cL) \rtimes \BC[W]
$$

\begin{example}
If we begin with  the trivial local system $\cL_0$ with its tautological $W$-equivariant structure, the endomorphisms of the resulting
Grothendieck-Springer sheaf  
$$S_{G, 0} \simeq \mu_{G!} \BC_{\wt G}[\dim G]
$$ 
are equivalent to the group algebra $\BC[W]$. 

 If we begin with the universal local system $\cL_{\univ}$ corresponding to the natural $W_{\aff}$-representation $\BC[\Lambda_H]$ and in turn to the structure sheaf $\cO_{H^\vee}$,
we obtain the universal Grothendieck-Springer sheaf 
$$S_{G, {\univ}} =  \mu_{G!} \nu_G^*  \cL_{univ}[\dim G]
$$
Although $S_{G, \univ}$ is not constructible, it is cohomologically bounded, and informally one can  view it as a perverse sheaf with endomorphisms equivalent to the group algebra $\BC[W_\aff]$. 
\end{example}

\begin{remark} 
Suppose the $W$-equivariant local system $\cL$ comes from a finite-dimensional $\BC[\Lambda_H]^W$-module,
 or in other words, a  coherent sheaf on $H^\vee\inv W$ with finite support. Then we may lift $\cL$ to a module over the Harish Chandra center $\mathfrak Z\fg \simeq \BC[\fh]^W$, or in other words,  to a coherent sheaf on $\fh^\vee\inv W$.
If we view this lift as a generalized eigenvalue for $\mathfrak Z\fg$, the $\D$-module corresponding to $S_{G, \cL}$ is equivalent to the Harish Chandra system of differential equations on $G$
where the  differential operators from $\mathfrak Z\fg$ are prescribed to act by this generalized eigenvalue.  
\end{remark}

In summary, the geometric Eisenstein series construction $\mu_{G!}\nu_{G}^*[\dim G]$ descends to a fully faithful embedding
$$
\xymatrix{
\Loc_\fin(H/W) \simeq \BC[W_{\aff}]\on{-mod}_\fin \simeq \on{Coh}_\fin(H^\vee/W) \ar@{^(->}[r] &  \on{Perv}(G/G)
}$$
with domain finite-rank $W$-equivariant local systems on $H$, or equivalently, finite-dimensional $W_{\aff}$-representations, or equivalently, $W$-equivariant coherent sheaves on $H^\vee$ with finite support, into $G$-equivariant perverse sheaves on $G$.

\begin{remark}
Going beyond abelian categories, one can keep track of the derived structure of $H$-equivariance, and also allow arbitrary $W$-equivariant constructible complexes on $H/H$.  The geometric Eisenstein series construction $\mu_{G!}\nu_{G}^*[\dim G]$ descends to a fully faithful embedding
$$
\xymatrix{
D^\flat_c((H/H)/W) \ar@{^(->}[r] & D^\flat_c(G/G)
}
$$ 
For example, the differential graded algebra of endomorphisms of $S_{G, 0}$ is the twisted product 
$H^*(H/H) \rtimes \BC[W]$, and that of $S_{G, \univ}$ is the twisted product
$H^*(BH) \rtimes \BC[W_{\aff}]$.

(See~\cite{sam} where the theory is in fact established for all $\D$-modules.) 
\end{remark}

\begin{remark}
Continuing with the derived structure of $H$-equivariance in the picture, the derived category of local systems on
$H/H \simeq H \times BH$ is equivalent to modules over $\BC[\Lambda_H] \otimes  \on{Sym}^*(\fh[1])$ via the identification $H_{-*}(H) \simeq \on{Sym}^*(\fh[1])$. Observe that  
the Langlands parameter moduli $\Loc_{H^\vee}(E_\node)$ of $H^\vee$-local systems on $E_\node$
admits the presentation
$$
\Loc_{H^\vee}(E_\node) \simeq H^\vee \times \Spec \on{Sym}^*(\fh[1]) \times BH^\vee
$$ 
Thus one can view the natural domain of the  geometric Eisenstein series embedding as $W$-equivariant 
coherent sheaves with finite support and trivial $H^\vee$-equivariant structure on  $\Loc_{H^\vee}(E_\node)$.
 (A similar interpretation applies in the rational setting discussed earlier, where the moduli  $\Loc_{H^\vee}(E_\cusp)$ is more simply the product $\Spec \on{Sym}^*(\fh[1]) \times BH^\vee$.)
\end{remark}


\subsection{Elliptic Springer theory} 
From a  geometric viewpoint, this paper extends the above story from the rational and trigonometric settings to the elliptic setting
of smooth elliptic curves. Via the restriction of geometric Eisenstein series to degree zero semistable bundles, we introduce an elliptic version of Grothendieck-Springer sheaves  and calculate their endomorphisms.
For degree zero semistable bundles in arithmetic genus one, the induction map from $B$-bundles to $G$-bundles is already proper, 
and so there is no need for the intricacies of Laumon or Drinfeld compactifications.
The Weyl group symmetry of the construction is a simple form of the functional equation and admits a straightforward verification.
The construction also works universally over the moduli of smooth elliptic curves (compatibility with the rational and trigonometric constructions at the boundary).

As discussed below with more detail, our primary motivation stems from understanding character sheaves for loop groups in the guise of perverse sheaves with nilpotent singular support on the moduli of $G$-bundles on a smooth elliptic curve $E$. Within this framework,
the Grothendieck-Springer sheaves produced by the following theorem form the {\em elliptic principal series},
or geometric avatars of principal series representations of loop groups. To simplify the story, let us assume that the derived
group $[G,G]\subset G$ is simply connected.

\begin{thm} For a smooth elliptic curve $E$, there is a fully faithful embedding
$$
\xymatrix{
\cS_E: \BC[W_E]\module_{fin}\ar@{^(->}[r] & \on{Perv}_\cN(Bun^{ss}_G(E))
}
$$
from finite-dimensional representations of the elliptic or  double affine Weyl group
 $$W_E = (\pi_1 (E)\otimes \Lambda_H) \rtimes W$$ 
 to perverse sheaves with nilpotent singular support on the moduli
  of semistable $G$-bundles on $E$.
  \end{thm}

\begin{remark}
The domain category $\BC[W_E]\module_{fin}$ of the theorem has two Langlands dual descriptions: on the one hand,  it is equivalent to finite-rank $W$-equivariant local systems on the automorphic moduli $\Bun_H^0(E)$ of degree zero $H$-bundles on $E$;
on the other hand, it is equivalent to  $W$-equivariant coherent sheaves with finite support and trivial  equivariance for automorphisms on the Langlands parameter moduli $\Loc_{H^\vee}(E)$. One can view the theorem as a small but interesting part of the geometric Langlands correspondence for the elliptic curve $E$.
\end{remark}

 \begin{remark}
Geometric Eisenstein series in genus one, as well as modified versions such as the elliptic Grothendieck-Springer sheaves of the theorem, are 
objects of the elliptic Hall category, introduced and studied in depth
by Schiffmann and Vasserot \cite{S1,S2,SV1,SV2,SV3}. Notably, the trace functions of geometric Eisenstein 
series are identified with Macdonald's symmetric functions. The elliptic Hall algebra, the Grothendieck group spanned by 
 geometric Eisenstein 
series,
is identified with a variant of Cherednik's double affine Hecke algebra and related to $K$-groups of Hilbert schemes of points.
One can view elliptic Springer theory as a categorical aspect of elliptic Hall algebras.
\end{remark}

\begin{remark}
With the results of this paper in hand, it is not difficult to adapt the arguments of the trigonometric case~\cite{sam} to understand the monadic structure of the elliptic Grothendieck-Springer construction and enhance the theorem to a statement on the level of derived categories. What results is a fully faithful embedding
$$
\xymatrix{
D^\flat_c(\Bun^0_H(E)/W) \ar@{^(->}[r] & D^\flat_c(\Bun^{ss}_G(E))
}
$$
which recovers the statement of the theorem
at the level of abelian categories of local systems and perverse sheaves.
\end{remark}

For a precise account of the functor $S_E$ and further details, the reader could proceed to Section~\ref{sect main}. We conclude the present introduction with an informal discussion of motivation from the theory of character sheaves, in particular 
for loop groups. This is not needed
to read Section~\ref{sect main} and the arguments in the subsequent sections.

  \subsubsection{Brief recollection about character sheaves}
Our interest in Springer theory stems from its central role in Lusztig's theory of character sheaves~\cite{character 1}.
In the traditional group-theoretic setting, the Grothendieck-Springer sheaves $S_{G, \lambda}$ attached to $W$-invariant characters $\lambda:\Lambda_H\to \BC^*$  provide the geometric avatars of principal series representations of finite groups of Lie type.

%

For each  $W$-invariant character $\lambda:\Lambda_H\to \BC^*$, the  category $\Ch^\lambda_G$ of character sheaves with central character $\lambda$ forms the categorical
Hochschild homology of the  monoidal $\lambda$-monodromic Hecke category $\cH^\lambda_G$. 
Any (sufficiently finite) $\cH^\lambda_G$-module category defines a character object in the categorical Hochschild homology    which is thus a character sheaf. The  Grothendieck-Springer sheaf $S_{G, \lambda}$ is the trace of the unit of $\cH^\lambda_G$, or equivalently, the character of the regular $\cH^\lambda_G$-module. The Grothendieck-Springer sheaves collectively provide the principal series character sheaves
underlying the 
principal series representations of finite groups of Lie type. More strikingly, every character sheaf appears as the trace of an endomorphism of a regular monodromic Hecke module category.  Thus the entire spectrum of finite groups of Lie type is captured by the {\em categorified principal series}.

From the perspective of topological field theory, one can view each monoidal monodromic Hecke category as a quantization of $G$-gauge fields on an interval with $B$-reductions at the end points. From general principles, one expects  its categorical Hochschild homology to be the analogous quantization of $G$-gauge fields on the circle, and thus the appearance of character sheaves 
is not surprising. 
But we will now turn to loop groups where this viewpoint leads to less evident
conclusions.


\subsubsection{Towards character sheaves for loop groups}
While Lusztig's character sheaves account for characters of finite groups of Lie types, for $p$-adic groups, such a theory is still not available but highly desirable. One obvious approach is to pass from $p$-adic groups to loop groups (from mixed to equal characteristic) and then  attempt to follow Lusztig's constructions. In particular, one might hope that the monoidal 
$\lambda$-monodromic affine Hecke categories $\cH^\lambda_{LG}$
would be rich enough to produce all depth zero characters. Unfortunately, with standard techniques, such a direct approach encounters serious obstructions in the infinite-dimensional and codimensional nature of adjoint orbits in loop groups. But as demonstrated most strikingly by Ng\^o in his proof of the Fundamental Lemma~\cite{Ngo},  Hitchin systems over global curves provide finite-dimensional models of this geometry.

Returning to the perspective of topological field theory, and in particular the Geometric Langlands program as a quantization of 
four-dimensional super Yang-Mills theory, one might look specifically to the Hitchin system for an elliptic curve $E$ to model the geometry  of the adjoint quotient of the loop group. Indeed, this is not a new idea: it has precedents in Looijenga's (unpublished) identification of holomorphic $G$-bundles on the Tate curve $E_q$ with twisted conjugacy classes in loop groups (see~\cite{EFK} and Baranovsky-Ginzburg's  refinement for semistable bundles and integral twisted conjugacy classes~\cite{BG}). Thus it is reasonable to look for a theory of character sheaves for loop groups in the geometry of the moduli $\Bun_G(E)$ of $G$-bundles on an elliptic curve $E$. 
(This  is discussed for example by Schiffmann~\cite{S2} who attributes the idea to Ginzburg).

With the preceding motivation and geometric reformulation of character sheaves~\cite{MV, ginzburg} in mind,
we propose the following definition.
 
\begin{defn} Given a smooth elliptic curve $E$, an {\em elliptic character sheaf} is a constructible complex 
on the moduli $\Bun_G(E)$
with singular support in the global nilpotent cone (the zero-fiber of the Hitchin system).
We denote the category of elliptic character sheaves by $\Ch_G(E)$. 
\end{defn}

While it appears difficult to construct a theory of character sheaves from the elliptic picture of twisted conjugacy classes, a more structured approach is available via the degeneration of the Tate curve $E_q \leadsto E_0$ to a nodal elliptic curve and then the passage to its normalization $\mathbb P^1 \simeq \tilde E_0$. This sequence offers an analogue of the horocycle transform for loop groups completely within the setting of finite-dimensional geometry.
Via the geometry of degenerations and normalizations, we expect to show~\cite{elliptic} that the categorical Hochschild homology of the $\lambda$-monodromic affine Hecke category $\cH^\lambda_{LG}$
 is equivalent to the category of elliptic character sheaves with central character $\lambda$. Any (sufficiently dualizable) $\cH^\lambda_{LG}$-module category has a character in the categorical Hochschild homology  which will then be such an elliptic character sheaf. 
 
Finally, the results of this paper independently produce the elliptic Grothendieck-Springer sheaf 
 $
 S_{E, \lambda}  \in D^\flat_c(\Bun^{ss}_G(E))
 $ 
 whose extension by zero is the elliptic character sheaf which  should arise  via the trace of the unit of $\cH^\lambda_{LG}$, or equivalently, the character of the regular $\cH^\lambda_{LG}$-module category.

\subsubsection{Langlands duality/Mirror symmetry}
Let us mention further developments  related to the preceding discussion.

Following Kazhdan-Lusztig \cite{KL, CG}, the affine Hecke algebra admits a Langlands dual presentation as equivariant coherent sheaves on the Steinberg variety of the dual group. Bezrukavnikov \cite{Roma ICM, Roma} categorifies this to an equivalence of the monoidal affine Hecke category (consisting of unipotent bimonodromic sheaves
on the affine flag variety) with equivariant coherent sheaves on the Steinberg variety. 

From this Langlands dual starting point, one can ask what are the geometric avatars of characters.  In joint work with A.~Preygel~\cite{BNP}, we provide the following answer for 
the ``global" version of the affine Hecke category, where we take into account all monodromies at once.
Recall that fixing a basis of $\pi_1(E)$ produces an equivalence from
 the moduli
 ${\Loc}_{G^\vee}(E)$ of dual group local systems on the elliptic curve $E$ to the derived stack of commuting pairs of dual group elements
 up to conjugacy: the two commuting elements are given by the monodromies of a local system around the basis of loops.

\begin{thm}[\cite{BNP}]
The Hochschild homology category of the global affine Hecke category is  equivalent to the derived category 
$D_\cN^\flat\Coh({\Loc}_{G^\vee}(E))$ of coherent sheaves 
with 
 nilpotent singular support
on the moduli
 of dual group local systems on the elliptic curve $E$. 
 \end{thm}

\begin{remark} 
Coherent sheaves with nilpotent singular support on moduli of local systems were introduced by~\cite{AG} as the natural target
for the spectral side of the geometric Langlands conjecture. 
\end{remark}

\begin{remark}
The commuting stack $\Loc_{G^\vee}(E)$ plays a 
central role in representation theory. We mention two  recent exciting developments. Its $K$-groups
play a central role (as the Langlands dual form of the elliptic Hall algebra) in the work of Schiffmann-Vasserot \cite{SV1,SV2,SV3},
with close ties to the theory of Macdonald polynomials and double affine Hecke algebras. The work of Ginzburg \cite{isospectral} unveils and exploits  a direct link between the (Lie algebra) commuting variety, Cherednik algebras and the Harish Chandra system or Springer sheaf.
\end{remark}

\begin{remark}
It is natural to ask which coherent sheaf on ${\Loc}_{G^\vee}(E)$ corresponds to the universal
elliptic Grothendieck-Springer sheaf, or in other words, the character of the regular ``global" affine Hecke module. 
Following through the constructions, one finds 
the {\em coherent Grothendieck-Springer sheaf} resulting from the pushforward of the
structure sheaf along the induction map $\Loc_{B^\vee}(E)\to  \Loc_{G^\vee}(E)$. (In fact, a slight modification
of this object matches where we take into account the natural  linearity over $H^\vee$.)

In joint work with D.~Helm~\cite{BHN}, we relate a natural $q$-deformation of this sheaf to the representation
theory of affine Hecke algebras, and more broadly, relate all coherent character sheaves with
representations of $p$-adic groups. We expect elliptic and coherent character sheaves to elucidate the depth zero representation theory of $p$-adic groups.

\end{remark}


\section{Main statements}\label{sect main}
Throughout the rest of the paper, 
we will make the simplifying assumption that the derived group $[G,G]\subset G$ is simply connected.
 

Fix an elliptic curve $E$ (or more precisely a smooth projective curve of genus one; our assertions will not involve any choice of a base point). The  geometry  of $G$-bundles on $E$, in particular the stack of  semistable bundles and its coarse moduli,  is well understood, with an extensive literature going back to Atiyah~\cite{A}. We have found the  beautiful sources ~\cite{R75, R96, La, FMW, FM}
to be particularly helpful.

Consider the  Eisenstein diagram of stacks of principal bundles
$$
\xymatrix{
\Bun_G(E) & \ar[l]_-{p_E} \Bun_B(E) \ar[r]^-{q_E} & \Bun_H(E)
}
$$

Recall that $\Bun_H(E)$ is naturally a commutative group-stack, and there is a canonical equivalence
$$
\Bun_H(E) \simeq \Pic(E)\otimes_\Z \Lambda_H 
$$
where $\Lambda_H = \Hom(\Gm, H)$ denotes the coweight lattice of $H$ and $Pic(E)$ denotes the Picard stack.
In particular, the group of connected components $\pi_0\Bun_H(E)$ is canonically isomorphic to $\Lambda_H$.
We  denote the neutral component of degree zero $H$-bundles by
$$
\cH_E \subset \Bun_H(E)
$$


Recall that the Weyl group $W$ naturally acts on $\Lambda_H$ and hence also on $\Bun_H(E)$ and hence in turn on $\cH_E $.
An $H$-bundle is said to be regular if it is not isomorphic to any of its $W$-translates.
We  denote the open substack of  $W$-regular degree zero $H$-bundles by
$$
\cH_E^{r} \subset \cH_E 
$$

Observe that $\pi_0\Bun_G(E)$ is canonically equivalent to $\Lambda_H/R_G$ where $R_G\subset \Lambda_H$ denotes the coroot lattice of $G$. We  say that a $G$-bundle is degree zero if it lies in the connected component of the trivial $G$-bundle. 
Consider   the open substack of degree zero semistable $G$-bundles
$$
\xymatrix{
\cG_E \subset \Bun_G(E)
}
$$
The coarse moduli of $S$-equivalent degree zero semistable $G$-bundles admits the description
$$
\xymatrix{
\cG_E /\text{\{$S$-equivalence\}}   \simeq \cH_E'\git W
}
$$
where $\cH_E'$ denotes the coarse moduli of degree zero $H$-bundles, and we take the geometric invariant theory quotient by $W$.

Consider the further open substack of regular semisimple degree zero semistable $G$-bundles
$$
\xymatrix{
\cG_E^{rs} \subset \cG_E
}
$$
Here regular means their automorphisms are as small as possible (of the dimension of $H$) and semisimple means they are induced from a torus bundle. There is a canonical equivalence
$$
\xymatrix{
\cG_E^{rs}  \simeq \cH_E^r/W
}
$$
where  the $W$-action on $ \cH_E^r$ is free (and so may be interpreted in any fashion).

%
%

\begin{example}
Take $G=SL_2$, so that $W\simeq \BZ/2\BZ$, $H\simeq \Gm$, and $\Lambda_H\simeq \BZ$. 

Then $\cH'_E \simeq E$, and the geometric invariant theory quotient $E\simeq \cH_E' \to \cH'_E\git W \simeq \mathbb P^1$ is the 
usual Weierstrass ramified two-fold cover. The moduli of semistable bundles lying above any of the four points of ramification is equivalent to the adjoint quotient of the nilpotent cone $\cN/G$, and the semistable bundles lying above the complement of the four points form the regular semisimple locus.
\end{example}

Recall that $q_E :\Bun_B(E) \to \Bun_H(E)$ is a bijection on connected components.
We  say that a $B$-bundle is degree zero if it projects to a degree zero $H$-bundle.
We  denote the connected component of degree zero $B$-bundles by
$$
\cB_E \subset \Bun_B(E)
$$
We have the natural restriction of the Eisenstein diagram
$$
\xymatrix{
\cG_E  & \ar[l]_-{\mu_E} \cB_E  \ar[r]^-{\nu_E} & \cH_E
}
$$
We refer to the map $\mu_E$ as the elliptic Grothendieck-Springer resolution.

Recall that a representable map $f:X\to Y$ of irreducible stacks is said to be small  if $f$ is proper, surjective, and for all $k>0$, we have $\on{codim} f(X_k) > 2k$ 
where $X_k\subset X$ denotes the union of fibers of $f$ of dimension $k$.

\begin{thm}\label{main thm}
(1) The elliptic Grothendieck-Springer resolution 
$$
\xymatrix{
\mu_E:\cB_E \ar[r] & \cG_E
}
$$ is a small map.

(2) The restriction of the  elliptic Grothendieck-Springer resolution  to the inverse-image $\cB_{E}^{rs}= \mu^{-1}_E(\cG_E^{rs})$ of the regular semisimple locus fits into the Cartesian diagram of $W$-covers
$$
\xymatrix{
\ar[d]_-\sim \cB_{E}^{rs} \ar[r]^-{\mu_E^{rs}} &\cG^{rs}_E \ar[d]^-\sim \\
\cH^r_E \ar[r] & \cH^r_E/W
}$$

\end{thm}

The proof of the theorem will be given in Section~\ref{sect proofs} after collecting some preliminaries in Section~\ref{sect prelim}.

\begin{remark} Some simple observations:

(1) The theorem and its proof extend universally over the moduli of elliptic curves.

(2) The theorem holds equally well over any projective curve of arithmetic genus one, and in fact the resulting objects are well-known  when no component is a smooth elliptic curve (so that the fundamental group of the curve is less than rank two). For example, over a cuspidal elliptic curve, one finds the Lie algebra version of the  Grothendieck-Springer resolution; and over  any nodal necklace of projective lines, in particular, a nodal elliptic curve, one finds
 the  group version.
 \end{remark}

The fundamental group $\pi_1 (\cH_E)$ is the tensor product  lattice $\pi_1(E) \otimes \Lambda_H$ with spectrum two copies of the dual torus $H^\vee \times H^\vee = \Spec \BC[\pi_1(E) \otimes \Lambda_H]$.  Thus (finite-rank) $W$-equivariant local systems on $\cH_E$ correspond to (finite-dimensional)
representations of the elliptic  Weyl group 
$$W_{E} = (\pi_1(E) \otimes \Lambda_H)\rtimes W
$$ which in turn correspond to 
$W$-equivariant quasicoherent sheaves on $H^\vee \times H^\vee$ (with finite global sections).
Starting from a finite-rank $W$-equivariant local system $\cL$, we define the corresponding elliptic Grothendieck-Springer sheaf by the formula
$$
S_{E, \cL} = \mu_{E!} \nu_E^*  \cL \in D^\flat_c(\cG_E)
$$

For the elliptic curve $E$, the Hitchin system reduces to the natural map $T^*\Bun_G(E) \to \fh^*\git W$ that assigns to a covector $\xi \in H^0(E, \fg^*_\cP)$ at a bundle $\cP\in \Bun_G(E)$ its characteristic polynomial (or equivalently, image under the geometric invariant quotient of the coadjoint action).
We define the nilpotent cone $\cN\subset  T^*\Bun_G(E)$ to be the zero-fiber of the Hitchin system.

\begin{corollary}\label{main cor}
Let $\cL$ be a finite-rank $W$-equivariant local system on $\cH_E$, or equivalently, a finite-dimensional $W_E$-representation.

(1)  The  elliptic Grothendieck-Springer sheaf $S_{E, \cL} $ is the middle-extension
of its restriction  to the regular semisimple locus.

(2)   The elliptic Grothendieck-Springer construction provides a fully faithful embedding
$$
\xymatrix{
\cS_E: \BC[W_E]\module_{fin}\ar@{^(->}[r] & \on{Perv}_\cN(Bun^{ss}_G(E))
}
$$
from finite-dimensional $W_E$-representations to perverse sheaves with nilpotent singular support.

In particular, 
the endomorphism algebra of $ S_{E, \cL}$ as a perverse sheaf is  equivalent to the twisted product $\End_{\Loc(\cH_E)}(\cL) \rtimes \BC[W]$.
\end{corollary}

The corollary follows from the  theorem except for the assertion about singular support; this   will be given in Section~\ref{sect proofs} 
with the proof of the theorem.

\begin{remark}
If we begin with the universal local system $\cL_{\univ}$ corresponding to the natural $W_{E}$-representation $\BC[\pi_1(E) \otimes \Lambda_H]$ and in turn to the structure sheaf of $H^\vee \times H^\vee$,
we obtain the universal Grothendieck-Springer sheaf 
$$
S_{E, \cL_{\univ}} =  \mu_{E!} \nu_E^*  \cL_{univ}
$$
Although $S_{E, \cL_\univ}$ is not constructible, it is cohomologically bounded, and informally one can  view it as a perverse sheaf with endomorphisms equivalent to the group algebra $\BC[W_{E}]$.
\end{remark}


\section{Tannakian reminders}\label{sect prelim}

Following the account in~\cite{BG}, we recall a useful way  to translate Borel structures on principal bundles into linear algebra.

We continue with the setup of $G$ a  connected reductive group with simply connected derived group, $B\subset G$ a Borel subgroup, $N\subset B$ its unipotent radical, and $H=B/N$
the universal Cartan.

Let $\check \Lambda_H= \Hom(H, \Gm)$ denote the weight lattice, and $\check\Lambda^+_G \subset\check \Lambda_H$ the cone of dominant weights. For  $\lambda \in \check \Lambda_G$, let $V^\lambda$ denote the  irreducible $G$-representation of highest weight $\lambda$.

Recall that  specifying the Borel subgroup $B\subset G$, or in other words a point of the flag variety of $G$, is equivalent to specifying a collection of lines 
$$
\xymatrix{
L^\lambda \subset V^\lambda, & \mbox{for $\lambda \in \check\Lambda_G$,}
}$$ satisfying the Plucker relations:  the natural  projections 
$$\xymatrix{
V^\lambda \otimes V^\mu \ar[r] & V^{\lambda+ \mu}, & \mbox{for $\lambda, \mu \in \check \Lambda_G$},
}$$ restrict to  isomorphisms 
$$
\xymatrix{
L^\lambda \otimes L^\mu \ar[r]^-\sim &  L^{\lambda + \mu}
}
$$

Now let $X$ be any base scheme. The moduli $\Bun_B(X)$ represents the following data: an $S$-point of $\Bun_B(X)$ is an $S$-point $\xi$ of $\Bun_G(X)$, together with a collection of invertible subsheaves
$$
\xymatrix{
\cL^\lambda \subset \cV_\xi^\lambda, & \mbox{for $\lambda \in \check\Lambda_G$,}
}$$ 
such that the quotients $\cV_\xi^\lambda/\cL^\lambda$ are flat over $S\times X$, and the collection satisfies the Plucker relations:  the natural  projections 
$$\xymatrix{
\cV_\xi^\lambda \otimes \cV_\xi^\mu \ar[r] & \cV_\xi^{\lambda+ \mu}, & \mbox{for $\lambda, \mu \in \check \Lambda_G$},
}$$ restrict to  isomorphisms 
$$
\xymatrix{
\cL^\lambda \otimes \cL^\mu \ar[r]^-\sim &  \cL^{\lambda + \mu}.
}
$$

Now suppose $X$ is a smooth connected projective curve. Unlike the absolute case over a point where the flag variety is connected and proper, the fibers of the map $p:\Bun_B(X) \to \Bun_G(X)$ need not be connected or proper. 
%
Following Drinfeld, to compactify the connected components of the fibers of  $p$, one can consider generalized $B$-structures. The moduli  $\ol \Bun_B(X)$ of generalized $B$-bundles represents the following data: an $S$-point of $\ol\Bun_B(X)$ is an $S$-point $\xi$ of $\Bun_G(X)$, together with a collection of invertible subsheaves
$$
\xymatrix{
\cL^\lambda \subset \cV_\xi^\lambda, & \mbox{for $\lambda \in \check\Lambda_G$,}
}$$ 
such that the quotients $\cV_\xi^\lambda/\cL^\lambda$ are flat relative to $S$ (but not necessarily over $S\times X$), and the collection satisfies the Plucker relations:  the natural  projections 
$$\xymatrix{
\cV_\xi^\lambda \otimes \cV_\xi^\mu \ar[r] & \cV_\xi^{\lambda+ \mu}, & \mbox{for $\lambda, \mu \in \check \Lambda_G$},
}$$ restrict to  isomorphisms 
$$
\xymatrix{
\cL^\lambda \otimes \cL^\mu \ar[r]^-\sim &  \cL^{\lambda + \mu}.
}
$$

The following statements can all be found in~\cite{BG}. There is an evident open embedding $\Bun_B(X) \subset \ol\Bun_B(X)$ whose image is dense. The  projection $q:\Bun_B(X) \to \Bun_H(X)$ extends to a  projection  $\ol q:\ol\Bun_B(X) \to \Bun_H(X)$. 
The  map $p:\Bun_B(X) \to \Bun_G(X)$ extends to a  representable map $\ol p :\ol\Bun_B(X) \to \Bun_H(X)$, 
and the restriction of $\ol p$ to any connected component is proper.


\section{Further arguments}\label{sect proofs}
Let $X$ be a smooth projective curve.

\begin{lemma}\label{no torsion} 
Suppose $\cL\subset \cV$ is the inclusion of a degree zero invertible sheaf into a degree zero  locally free sheaf. If $\cV$ is semistable, then $\cV/\cL$ is locally free,
and so $\cL\subset \cV$ is an inclusion of vector bundles.
\end{lemma}

\begin{proof}
If $\cV/\cL$ contains torsion, then we may twist $\cL$ to obtain a positive locally free subbundle of $\cV$, a contradiction to the fact that $\cV$ is  degree zero and semistable.
\end{proof}

Let $\cG_{X}$ be the moduli of degree zero, semistable $G$-bundles on $X$. Let
$\cB_X$ be the moduli of degree zero $B$-bundles on $X$, where by degree we mean the degree of the induced $H$-bundle.

\begin{lemma}\label{proper}  The induction map
$
\mu_X:\cB_X \to \cG_X
$ is representable and proper.
\end{lemma}

\begin{proof}
The Plucker presentation shows $\mu_X$ is representable.  Lemma~\ref{no torsion} then shows it is proper. 
\end{proof}

Let $x\in X$ be a geometric point.

\begin{lemma}\label{determined at point}
Suppose $\cV$ is a degree zero semistable vector bundle, and $\cL_1, \cL_2\subset \cV$ are degree zero  line subbundles.  
If $\cL_1|_x = \cL_2|_x$, then $\cL_1 = \cL_2$.
\end{lemma}

\begin{proof}
Suppose not. Then the composition $\cL_1 \subset \cV \to \cV/\cL_2$ is an inclusion of a degree zero invertible sheaf into a degree zero semistable locally free sheaf. By Lemma~\ref{no torsion}, the composition must be an inclusion of vector bundles, a contradiction to the fact that it must also have a zero at $x$.
\end{proof}

Let $\cG_{X, x} = \cG_E \times_{BG} BB$ be the moduli of a degree zero semistable $G$-bundle $\xi$  together with a flag in the fiber $\xi|_x$. 

\begin{lemma}\label{fiber flag}
The natural map $r_X:\cB_X \to \cG_{X, x}$ is a closed embedding.
\end{lemma}

\begin{proof}
Apply Lemma~\ref{determined at point} to the Plucker presentation.
\end{proof}

Let $\cZ_X = \cB_X \times_{\cG_X} \cB_X$ be the moduli of a pair of degree zero $B$-bundles together with an isomorphism of their induced $G$-bundles. 

\begin{lemma}\label{ell steinberg}
If  $X$ is  an elliptic curve, then  the irreducible components of $\cZ_X$ are in natural bijection with the Weyl group $W,$ and 
the dimension of each
irreducible component is  zero. 
\end{lemma}

\begin{proof}
Fix $x\in X$ a geometric point, and consider the natural restriction map $\pi_X:\cZ_X \to BB \times_{BG} BB \simeq B\bs G/B.$ By Lemma~\ref{fiber flag}, for a given relative position $Y_w \subset B\bs G/B$ represented by two Borel subgroups $B_1 , B_2\subset G$, the inverse image $\pi_X^{-1}(Y_w)$ is equivalent to the moduli of degree zero ${B_1\cap B_2}$-bundles. It is simple to check this moduli has the correct dimension by induction on the solvable filtration of ${B_1\cap B_2}$. One need only use that the canonical bundle of $X$ is trivial.
\end{proof}

Consider the composition $q_{1}:\cZ_X \to \cB_X \to \cH_X$ where the first map is projection along the first factor and the second is the usual projection.

\begin{lemma}\label{spread out}
The restriction of $q_1$ to each irreducible component of $\cZ_X$ has equidimensional fibers.
\end{lemma}

\begin{proof}
Follows from the proof of Lemma~\ref{ell steinberg} where the components are shown to be moduli of degree zero ${B_1\cap B_2}$-bundles. 
\end{proof}

A representable map $f:X\to Y$ of irreducible stacks is said to be small (or semi-small) if $f$ is proper, surjective, and for all $k>0$, we have $\on{codim} f(X_k) > 2k$ (or $\on{codim} f(X_k) \geq 2k$)
where $X_k\subset X$ denotes the union of fibers of $f$ of dimension $k$. 

We will employ the following general strategy 
to establish a map $f:X\to Y$ is small.

First, suppose $f$ is proper, surjective,   and $V\subset Y$ is an open substack such that $f:U = X \times_Y V \to V$ is finite.
Let $X'= X\setminus U$, $Y' = Y \setminus V$, and $g:X'\to Y'$ be the restriction of $f$. It is immediate from the definitions that if $g$ is semi-small, then $f$ is small. 

Second, 
to see $g:X' \to Y'$ is semi-small, it suffices to show that the dimension of each irreducible component of $Z = X'\times_{Y'} X'$ is less than or equal to $\dim Y$.

\begin{proof}[Proof of Theorem~\ref{main thm}]
Assertion (2) is straightforward to check.  
Then Lemmas~\ref{proper}, \ref{ell steinberg}, and \ref{spread out} show that the restriction of $\mu_X$ to the complement of the regular semisimple locus is semi-small. This then in turn establishes  assertion (1).
\end{proof}

\begin{proof}[Proof of Corollary~\ref{main cor}]
It only remains to establish the assertion about singular support. At a $B$-bundle $\cP$, the pullback of covectors along $\mu_E$ is given by 
$H^0(E, \fg^*_\cP) \to H^0(E, \fb^*_\cP)$. The Hitchin system factors through the natural projection $H^0(E, \fb^*_\cP)\to \fh^*$. 
Thus the construction $S_E$ will take local systems, i.e.~sheaves with vanishing singular support on $\cH_E$, to perverse sheaves with nilpotent singular support on $\cG_E$.
\end{proof}

\section{Acknowledgements}
We would like to thank Sam Gunningham for explaining his viewpoint on Springer theory and Eisenstein series.
We would also like to  thank Dennis Gaitsgory and Tom Nevins for stimulating discussions about Eisenstein series and representation
theory in an elliptic setting, respectively. We gratefully acknowledge the support of NSF grants DMS-1103525 (DBZ)
and DMS-1319287 (DN).


\end{document}